\documentclass[12pt]{scrartcl}
 \usepackage{cmap}
\usepackage{amsmath,amsxtra,amscd,amssymb,latexsym,stmaryrd,amsthm,mathrsfs}
\usepackage{mathrsfs,mathtools,xcolor}
\usepackage{lmodern} % ou ae, aeguill ?
  \usepackage[utf8]{inputenc}
  \usepackage[T1]{fontenc}

\usepackage[colorlinks=true,linkcolor=red!70!black,citecolor=green!70!black,
    urlcolor=magenta!70!black,backref]{hyperref}

\addtokomafont{title}{\rmfamily\itshape}

\theoremstyle{plain}
\newtheorem{theo}{Theorem}[section]
\newtheorem*{theo*}{Theorem}
\newtheorem{coro}[theo]{Corollary}
\newtheorem{prop}[theo]{Proposition}
\newtheorem{lemm}[theo]{Lemma}
\newtheorem{theomain}{Theorem}
\newtheorem{coromain}[theomain]{Corollary}

\theoremstyle{definition}

\newtheorem{rema}[theo]{Remark}

\newcommand*{\dd}%
  {\relax\ifnum\lastnodetype>0\mskip\medmuskip\fi\mathrm{d}}
\newcommand{\pr}{\operatorname{\mathbb{P}}}
\newcommand{\esp}{\operatorname{\mathbb{E}}}
\newcommand{\fspace}[1]{\mathscr{#1}}

\newcommand{\op}[1]{\mathrm{#1}}
\newcommand{\one}{\boldsymbol{1}}

\newcommand{\Hol}{\operatorname{Hol}}
\newcommand{\wass}{\operatorname{W}}
\newcommand{\Lip}{\operatorname{Lip}}
\newcommand{\diam}{\operatorname{diam}}
\newcommand{\osc}{\operatorname{osc}}

\newcommand{\Ck}[1]{\operatorname{\mathcal{C}^{#1}}}
\newcommand{\Cku}[1]{\operatorname{\mathcal{C}^{#1}_1}}
\newcommand{\proba}{\operatorname{\mathcal{P}}}
\newcommand{\Var}{\operatorname{Var}}

\title{Empirical measures: regularity is a counter-curse to dimensionality}

\author{Beno\^{\i}t R. Kloeckner \thanks{Universit\'e Paris-Est, Laboratoire d'Analyse et de Mat\'ematiques Appliqu\'ees (UMR 8050), UPEM, UPEC, CNRS, F-94010, Cr\'eteil, France}}

\begin{document}
%%%%%%%%%%%%%%%%%%%%%%%%%%%%%%%%%%%%%%%%%%%%%%%%%%%%%%%%%%%%%%%
%%%%%%%%%%%%%%%%%%%%%%%%%%%%%%%%%%%%%%%%%%%%%%%%%%%%%%%%%%%%%%%
%%%%%%%%%%%%%%%%%%%%%%%%%%%%%%%%%%%%%%%%%%%%%%%%%%%%%%%%%%%%%%%

\maketitle

\begin{abstract}
We propose a ``decomposition method'' to prove non-asymptotic bound for the convergence of empirical measures in various dual norms. The main point is to show that if one measures convergence in duality with sufficiently regular observables, the convergence is much faster than for, say, merely Lipschitz observables. Actually, assuming $s$ derivatives with $s>d/2$ ($d$ the dimension) ensures an optimal rate of convergence of $1/\sqrt{n}$ ($n$ the number of samples). The method is flexible enough to apply to Markov chains which satisfy a geometric contraction hypothesis, assuming neither stationarity nor reversibility, with the same convergence speed up to a power of logarithm factor. 

Our results are stated as controls of the expected distance between the empirical measure and its limit, but we explain briefly how the classical method of bounded difference can be used to deduce concentration estimates.
\end{abstract}

%%%%%%%%%%%%%%%%%%%%%%%%%%%%%%%%%%%%%%%%%%%%%%%%%%%%%%%%%%%%%%%
%%%%%%%%%%%%%%%%%%%%%%%%%%%%%%%%%%%%%%%%%%%%%%%%%%%%%%%%%%%%%%%
\section{Introduction}

%%%%%%%%%%%%%%%%%%%%%%%%%%%%%%%%%%%%%%%%%%%%%%%%%%%%%%%%%%%%%%%
\subsection{Empirical measures and quadrature}

Consider a discrete-time stochastic process $(X_k)_{k\ge 0}$ taking its values in some phase space $\Omega$, assumed to be a Polish space endowed with its Borel $\sigma$-algebra. We are concerned with the random atomic measure \[ \hat\mu_n = \frac1n \sum_{k=1}^n \delta_{X_k}, \]
called the \emph{empirical measure} of the process, and its convergence. We shall either assume that the $(X_k)_{k\ge0}$ are independent identically distributed of some law $\mu$, or assume some weak long-range dependence and convergence of the law of $X_k$ to $\mu$ as $k\to\infty$.

To quantify the convergence, we are interested in distances on the set $\proba(\Omega)$ of probability measures defined by duality. Given a class $\fspace{F}$ of functions $f:\Omega\to\mathbb{R}$ (sometime called ``test functions'' or ``observables''), one defines for $\nu_0,\nu_1 \in \proba(\Omega)$:
\[ \lVert \nu_0-\nu_1\rVert_{\fspace{F}} = \sup_{f\in\fspace{F}} \big\lvert \nu_0(f)-\nu_1(f) \big\rvert\]
(note that we write indifferently $\nu_0(f)$ or $\int f \dd\nu_0$).

One particularly important case is obtained by taking $\fspace{F}=\Lip_1(\Omega)$, the set of $1$-Lipschitz functions. The corresponding metric is the $1$-Wasserstein metric $\wass_1= {\lVert \cdot \rVert_{\Lip_1}}$, which by virtue of \emph{Kantorovich duality} can be written equivalently as
\[\wass_1(\nu_0,\nu_1) := \inf_{X\sim \nu_0, Y\sim \nu_1} \esp \big[ \lVert X-Y\rVert \big]\]
where $\lVert \cdot\rVert$ here is the Euclidean norm and the infimum is over all pairs of random variable with the given measures as individual laws.
It is long-known \cite{ajtai1984optimal} that, when the $(X_k)_{k\ge0}$ are independent and uniformly distributed on $[0,1]^d$ , we have
\begin{equation}
\esp\big[ \wass_1(\hat\mu_n,\lambda) \big] \asymp
\begin{dcases*} 
\frac{1}{\sqrt{n}} & if $d=1$, \\[2\jot]
\sqrt{\frac{\log n}{n}} & if $d=2$, \\[2\jot]
\frac{1}{n^{\frac1d}} & if $d\ge 3$.
\end{dcases*}
\label{eq:speed}
\end{equation}
where $\asymp$ expresses upper and lower bounds up to multiplicative constants and $\lambda$ denotes the Lebesgue measure. This problem and generalizations have been studied in several works, e.g. \cite{talagrand1992matching, talagrand1994sharper, boissard2014mean, dereich2013constructive, fournier2015rate, ambrosio2016pde, weed2017sharp}.

The bounds \eqref{eq:speed} are interesting theoretically, but are rather negative for the practical application to quadrature. 
Computations of integrals are in many cases impractical using deterministic methods, and one often has to resort to Monte Carlo methods, i.e. approximate the unknown $\mu(f)$ by $\hat\mu_n(f)$. When one has to compute the integrals of a large number of functions $(f_m)_{1\le m \le M}$  with respect to a fixed measure $\mu$, one would rather draw the random quadrature points $X_1,\dots, X_k$ once and for all, and use them for all functions $f_m$; while usual Monte Carlo bound will ensure each individual estimate $\hat\mu_n(f_m)$ has small probability to be far from $\mu(f_m)$, if $M$ is large compared to $n$ these bounds will not ensure that \emph{all} estimates are good with high probability. On the contrary, convergence in $\wass_1$ (or in duality with some other class $\fspace{F}$) ensures good estimates simultaneously for all $f_m$, as long as they belong to the given class, independently of $M$. This makes such convergence potentially useful; but the \emph{rate} given above, $n^{-\frac1d}$, is hopelessly slow in high dimension which is precisely the setting where Monte Carlo methods are most needed. We shall prove that if the functions of interest are regular, then this ``curse of dimensionality'' can be overcome. We shall be interested in the duality with $\Cku{s}$ the set of functions with $\Ck{s}$ norm at most $1$ (precise definitions are given below; when $s=1$ this is the set of $1$-Lipschitz functions); but other spaces could be considered, e.g. Sobolev or Besov spaces.

Another issue is that in many cases, drawing independent samples $(X_k)_{k\ge 0}$ of law $\mu$ is not feasible, and one is lead to instead rely on a Markov chain having $\mu$ as its stationary measure; this is the Markov Chain Monte Carlo method (MCMC). While the empirical measure of Markov chains have been considered by Fournier and Guillin \cite{fournier2015rate}, these authors need quite strong assumptions: a spectral gap in the $L^2$ space (or similarly large spaces), and a ``warm start'' hypothesis ($X_0$ should have a law absolutely continuous with respect to $\mu$). In good cases, one can achieve this by a burn-in period (start with arbitrary $X_0$, and consider $(X_{k_0+k})_{k\ge 0}$ for some large $k_0)$; but in some cases, each $X_k$ has a singular law with respect to $\mu$ (for example the natural random walk generated by an Iterated Function System). We shall consider Markov chains satisfying a certain geometric contraction property, but again the method can certainly be adapted to other assumptions.

%%%%%%%%%%%%%%%%%%%%%%%%%
\subsection{Markov chains}

Our main result handles Markov chains of arbitrary starting distribution and with a spectral gap in $\Lip$ (e.g. positively curved chains in the sense of Ollivier \cite{ollivier2009ricci}).
\begin{theomain}\label{theomain:Markov}
Assume that $(X_k)_{k\ge0}$ is a Markov chain defined on a bounded domain $\Omega$ of $\mathbb{R}^d$, whose iterated transition kernel $(m^t_x)_{x\in\Omega,t\in\mathbb{N}}$ defined by 
\[ m_x^t(A) = \pr(X_{k+t}\in A \mid X_k=x)\]
is exponentially contracting in the Wasserstein metric $\wass_1$, i.e. there are constants $D\ge 1$ and $\theta\in(0,1)$ such that
\[ \wass_1(m_x^t,m_y^t) \le D\theta^t \lVert x-y\rVert. \]
Denote by $\mu$ the (unique) stationary measure of the transition kernel.

Then for some constant $C=C(\Omega,d,D,s)$ and all large enough $n$, letting $\bar n=(1-\theta)n$, we have
\begin{equation}
\esp\big[ \lVert \hat\mu_n - \mu \rVert_{\Cku{s}} \big] \le C \begin{dcases*}
\frac{(\log \bar n)^{\frac{d}{2s+1}}}{\sqrt{\bar n}} & when $s > d/2$\\[2\jot]
\frac{\log \bar n}{\sqrt{\bar n}} & when $s=d/2$ \\[2\jot]
\frac{(\log \bar n)^{d-2s+\frac sd}}{\bar n^{\frac sd}} & when $s < d/2$
\end{dcases*}
\label{eq:theo-Markov}
\end{equation}
\end{theomain}
Let us stress two strengths of this result:
\begin{itemize}
\item for $s=1$, recalling $\lVert \cdot\rVert_{\Cku{1}}=\lVert \cdot\rVert_{\Lip_1}=\wass_1$, the bounds are only a power of logarithm factor away from the optimal bounds for IID random variables,
\item for $s$ large enough, we almost obtain the optimal convergence rate $\asymp 1/\sqrt{n}$
\item we assume neither reversibility, stationarity, nor warm start hypotheses (the distribution of $X_0$ can be arbitrary),
\item the rate of convergence does not depend on the specific feature of the Markov chain, only on $D$ and $\theta$.
\end{itemize}
Note that for fixed $\theta$, $\bar n$ has the same order than $n$, but if $\theta$ is close to $1$, $1/(1-\theta)$ is the typical time scale for the decay of correlations. One thus cannot expect less than $(1-\theta)n$ Markov samples to achieve the bound obtained for $n$ independent samples.

Examples of Markov chains which are exponentially contracting in $\wass_1$ (equivalently, that have a spectral gap in the space of Lipschitz observables) are numerous; it is a slightly more general condition than ``positive curvature'' in the sense of Ollivier \cite{ollivier2009ricci}, see e.g. \cite{joulin2010curvature} and \cite{K:concentration} for concrete examples, or in the context of dynamical systems \cite{kloeckner2015contraction} and \cite{kloeckner2017optimal}.

Under the assumption of Theorem \ref{theomain:Markov}, it is well-known that uniform estimates
\begin{equation}
\sup_{f\in\mathscr{F}} \pr\big(\lvert \hat\mu_n(f)-\mu(f)\rvert >\varepsilon\big) \to 0 \qquad\text{and}\qquad \sup_{f\in\mathscr{F}}\esp\big[ \lvert \hat\mu_n(f)-\mu(f)\rvert \big] \to 0
\label{eq:convergence}
\end{equation}
hold, here with $\mathscr{F}=\Lip_1$ (or any smaller class), with a Gaussian rate.

The problem of convergence in duality to the class $\mathscr{F}$ is thus to invert the supremum and the probability (or expectancy), to bound from above
\[\pr\big(\sup_{f\in\mathscr{F}} \lvert \hat\mu_n(f)-\mu(f)\rvert >\varepsilon\big) \qquad\text{or}\qquad \esp\big[ \sup_{f\in\mathscr{F}} \lvert \hat\mu_n(f)-\mu(f)\rvert \big].\]
We shall disregard the potential issue of non-measurability: as we shall only deal with classes $\mathscr{F}$ having a countable subset which is dense in the uniform norm, we can always replace the supremum with a supremum over a countable set of functions.

The idea of the proof of Theorem \ref{theomain:Markov} is to take an arbitrary $f\in\Cku{s}(\Omega)$ and decompose it using Fourier series. The regularity hypothesis gives us a control on both the uniform approximation by a truncated Fourier series, and on the Fourier coefficients. Combining these controls, we bound from above $\lvert \hat\mu_n(f)-\mu(f)\rvert$ by a quantity that does not depend on $f$ at all, but depends on the Fourier basis elements $(e_k)_{k\in\mathbb{Z}^d}$ up to some index size. Taking a supremum and an expectation, this leaves us with the simple task to optimize where to truncate the Fourier series.

This decomposition method can in principle be used under various assumptions on the process $(X_k)_{k\ge0}$, the point being to identify a decomposition suited to the assumption; in particular, one can easily adapt the method to study geometrically ergodic Markov chains. I chose to present Theorem \ref{theomain:Markov} in part because its hypothesis is relevant to several Markov chains I am interested in, and in part because it presents specific difficulties:  a blunt computation leads to non-optimal powers of $n$. To obtain good rates, we translate the contraction hypothesis to frame part of the argument in the space $\Hol_\alpha$, where the Fourier basis has smaller norm; and instead of bounding the Fourier coefficients of a Lipschitz function directly, we use Parceval's formula and the injection $\Ck{s}\to H^s$ which turns out to give a better estimate. Another functional decomposition, and another path in computations might improve the power in the logarithmic factor.

We restrict to the compact case, but the method can in principle be adapted, or truncation argument be used, to deal with non-compactly supported measure. 

In order to introduce the decomposition method and show its flexibility, we shall state two simpler results below.

%%%%%%%%%%%%%%%%%%%%%%%%%%
\subsection{Explicit bounds in the i.i.d case, for the Wasserstein metric}

The decomposition method enables one to get a very explicit version of \eqref{eq:speed} with a few computations but very little sophistication.
\begin{theomain} \label{theomain:W1}
If $\mu$ is any probability measure on $[0,1]^d$ and $(X_k)_{k\ge 0}$ are i.i.d. random variable with law $\mu$, then for all $n\in\mathbb{N}$ we have
\begin{equation}
\esp\big[\wass_1(\hat\mu_n,\mu) \big] \le \begin{dcases*}
\frac{1}{2(\sqrt{2}-1)}\cdot \frac{1}{\sqrt{n}} & when $d=1$\\[2\jot]
\frac{\log_2(n)+8}{\sqrt{8n}} & when $d=2$ \\[2\jot]
\frac{C_d}{n^{\frac1d}} & when $d\ge 3$
\end{dcases*}
\label{eq:theo-W1}
\end{equation}
where $C_3\le 6.3$, $C_d \le 3 \sqrt{d}$ for all $d\ge 4$, and $C_d/\sqrt{d} \to 2$ as $d\to\infty$.
\end{theomain}

The order of magnitude of these bounds is sharp in many regimes:
\begin{itemize}
\item in dimension $1$, the order of magnitude $1/\sqrt{n}$ is optimal; however the constant $1/(2(\sqrt{2}-1))$ is \emph{not} asymptotically optimal when $\mu$ is Lebesgue measure,
\item when $d=2$ and $\mu$ is Lebesgue measure, as previously mentioned the correct order is $\sqrt{\log n/n}$, but to the best of my knowledge it is an open question to determine whether this better order holds for arbitrary measures (a positive answer is strongly expected). See Section \ref{sec:four-corners} for an example showing that in a more general setting the order $\log n/\sqrt{n}$ cannot be improved,
\item when $d\ge 3$, both orders of magnitude $n^{-1/d}$ as $n\to\infty$ and $\sqrt{d}$ as $d\to\infty$ are sharp up to multiplicative constants  (see Remark \ref{rema:lowerbound}). The asymptotic constant $2$ is certainly quite larger than the asymptotic constant
\[ \lim_{d\to\infty} \lim_{n\to\infty} \frac{n^{\frac1d}}{\sqrt{d}} \esp\big[\wass_1(\hat\mu_n,\lambda) \big] \]
which has been computed for the related, but slightly different \emph{matching problem} by Talagrand \cite{talagrand1992matching}; but our bound holds for all $n$ and all $d$ (and also all $\mu$). An even more general bound has been given by Boissard and Le Gouic \cite{boissard2014mean}, but their constant is larger by a factor approximately $10$.
\end{itemize}

Let us stress that the main purpose of this result will be to expose our method in an elementary setting: indeed many previous similar bounds are available in this case. For example more general non-asymptotic results have been obtained by Fournier and Guillin \cite{fournier2015rate}, building on previous work by Dereich, Scheutzow and Schottstedt \cite{dereich2013constructive}. They are more general in that they consider $q$-Wasserstein metric for any $q>0$ (while we will only be able to consider $q\le 1$), and apply to non-compactly supported measures $\mu$ under moment assumptions. However their constants, though non-asymptotic, have not been made explicit, and their behavior when the dimension grows has not been studied.

%%%%%%%%%%%%%%%%%%%%%%%%%%%%%%%%%%%%%%%
\subsection{Regular observables and independent samples}

In the i.i.d. case, we can improve Theorem \ref{theomain:Markov} by removing most of the logarithmic factors.
\begin{theomain}\label{theomain:reg}
If $\mu$ is any probability measure on $[0,1]^d$ and $(X_k)_{k\ge 0}$ are i.i.d. random variable with law $\mu$, then for all $s\ge 1$, for some constant $C=C(d,s) > 0$ (not depending upon $\mu$),  and all integer $n\ge 2$ we have
\begin{equation}
\esp\big[\lVert \hat\mu_n -\mu \rVert_{\Cku{s}} \big] \le C
\begin{dcases*}
\frac{1}{\sqrt{n}} & when $s > \frac{d}{2}$ \\[2\jot]
\frac{\log n}{\sqrt{n}} & when $s = \frac{d}{2}$ \\[2\jot]
\frac{1}{n^{s/d}} & when $s < \frac{d}{2}$
\end{dcases*} 
\label{eq:theo-reg}
\end{equation}
\end{theomain}

It is possible to prove this result with previous, more classical methods. Indeed, combining the ``entropy bound'' for the class $\Cku{s}$ \cite[Thm 2.7.1]{vaart1996weak} and the ``chaining method'' (see e.g. \cite[Ex 5.11, p. 138]{handel2016probability}) leads to Theorem \ref{theomain:reg}; I am indebted to Jonathan Weed for pointing this out to me. The proof by the decomposition method we provide here is very simple, but non-elementary as it relies on a wavelet decomposition. It is well-known that all functions in $\Cku{s}$ can be written as a linear combination of a few elements of a wavelet basis, with small coefficients, up to a small error. Then controlling $\lvert \hat\mu_n(f)-\mu(f)\rvert$ for all $f\in \Cku{s}$ simultaneously reduces to controlling this quantity for the few needed elements of the wavelet basis.

%%%%%%%%%%%%%%%%%%%%%%%%%%%%%%%%%%%%%%%
\subsection{concentration inequalities}

Up to know, we have restricted to estimates on the expectancy, while in many practical situations one would need concentration estimates. This is in fact not a restriction, as we shall explain briefly in Section \ref{sec:conc}: the classical bounded difference method enable one to get concentration near the expectancy. In particular, we get the following.
\begin{coromain}\label{coromain:conc}
Under the assumptions of Theorem \ref{theomain:Markov}, for some $\epsilon$ depending on $\theta,D,\diam \Omega$, for all large enough $n$ and all $M\ge C=C(\Omega, d, D, \theta)$ we have:
\begin{itemize}
\item when $s>d/2$
\begin{equation}
\pr\Bigg[ \lVert \hat\mu_n - \mu \rVert_{\Cku{s}} \ge M \frac{(\log n)^{\frac{d}{2s+1}}}{\sqrt{n}} \Bigg] \le e^{-\epsilon (M-C)^2(\log n)^{\frac{d}{2s+1}}}
\end{equation}
\item when $s=d/2$
\begin{equation}
\pr\Bigg[ \lVert \hat\mu_n - \mu \rVert_{\Cku{s}} \ge M \frac{\log n}{\sqrt{n}} \Bigg] \le e^{-\epsilon(M-C)^2 (\log n)^2}
\end{equation}
\item when $s < d/2$
\begin{equation}
\pr\Bigg[ \lVert \hat\mu_n - \mu \rVert_{\Cku{s}} \ge M \frac{(\log n)^{d-2s+\frac sd}}{n^{\frac sd}} \Bigg] \le e^{-\epsilon(M-C)^2 n^{1-2s/d}}.
\end{equation}
\end{itemize}
\end{coromain}
(The last inequality is not optimal as we relaxed the poly-logarithmic factor for simplicity.)

For example, when $s\ge d/2$ we deduce that
$ \frac{\sqrt{n}}{\log n} \lVert \hat\mu_n-\mu\rVert_{\Cku{s}} $
is bounded almost surely.

\paragraph{Structure of the paper}
Sections \ref{sec:W1}, \ref{sec:main} and \ref{sec:MC} are independent and contain the proofs of the main Theorems (\ref{theomain:W1}, \ref{theomain:reg} and \ref{theomain:Markov} respectively: we start with the most elementary proof, follow with the simplest one, and end with the most sophisticated).

Section \ref{sec:conc}, dealing with concentration estimates, is mostly independent from the previous ones, which are only used to deduce Corollary \ref{coromain:conc}.

We shall write $a \lesssim b$ for $a \le C b$, the dependency of the constant $C$ being left implicit unless it feels necessary; the constants denoted by $C$ will be allowed to change from line to line.

%%%%%%%%%%%%%%%%%%%%%%%%%%%%%%%%%%%%%%%%%%%%%%%%%%%%%%%%%%%%%%%
%%%%%%%%%%%%%%%%%%%%%%%%%%%%%%%%%%%%%%%%%%%%%%%%%%%%%%%%%%%%%%%
\section{Wasserstein convergence and dyadic decomposition}\label{sec:W1}

The goal of this Section is to prove (a refinement of) Theorem \ref{theomain:W1}.
We consider a sequence $(X_k)_{1\le k}$ of independent, identically distributed random points whose common law shall be denoted by $\mu$; we assume that $\mu$ is supported on the cube $[0,1]^d$ and consider the convergence of the empirical measure $\hat\mu_n := \sum_{k=1}^n \frac1n \delta_{X_k}$ in the $q$-Wasserstein distance where $q\in(0,1]$, i.e.
\[ \wass_{q}(\mu_0,\mu_1) := \inf_{f\in \Hol^{q}_1} \big\lvert \mu_0(f) - \mu_1(f) \big\rvert  \]
where $\Hol^{q}_1$ is the set of functions $f: [0,1]^d \to \mathbb{R}$ such that for all $x,y\in [0,1]^d$:
\[ \lvert f(x) - f(y) \rvert \le \lVert x-y\rVert^q  \]

While we are mostly interested in the Euclidean norm $\lVert\cdot\rVert$, our method is sharper in the case of the supremum norm\footnote{The same notation is used for the uniform norm of functions, but the type of the argument will prevent any confusion.} $\lVert\cdot\rVert_\infty$, with respect to which the analogue of the aforementioned objects are denoted by $\wass_{q,\infty}$ and $\Hol^{q,\infty}_1$. We will work with $\lVert\cdot\rVert_\infty$, and then deduce directly the corresponding result for the Euclidean norm by using that $\lVert\cdot\rVert \le \sqrt{d} \lVert\cdot\rVert_\infty$ (and thus $\wass_{q} \le d^{\frac{q}{2}}\wass_{q,\infty}$).

Our most precise result is the following.
\begin{theo}\label{theo:Wq}
For all $q\in(0,1]$ and all $n$, it holds:
\[\esp\big[ \wass_{q,\infty}(\hat\mu_n,\mu) \big] \le 
\begin{dcases*}
\frac{2^{\frac{d}{2}-2q}}{1-2^{\frac{d}{2}-q}}\cdot\frac{1}{\sqrt{n}} & when $d < 2q$,\\[2\jot]
\Big(2 + \frac{\log_2(n)}{2^{q+1}q} \Big) \frac{1}{\sqrt{n}} & when $d = 2q$ \\[2\jot]
2 \Big(\frac{\frac{d}{2}-q}{2q(1-2^{q-\frac{d}{2}})}\Big)^{\frac{2q}{d}} \Big( 1  + \frac{q}{2^q(\frac{d}{2}-q)} \Big) \frac{1}{n^{\frac{q}{d}}} & when $d > 2q$.
\end{dcases*}
\]
\end{theo}
We deduce several more compact formulas below, including Theorem \ref{theomain:W1}. Observe that for fixed $q$ and large $d$, the complicated front constant converges to $2$.

\begin{rema}\label{rema:lowerbound}
It is not difficult to see that for $\mu$ the Lebesgue measure and an optimal, deterministic approximation $\tilde\mu_n$ with $n=k^d$ Dirac masses, one has 
\[\wass_{1,\infty}(\tilde\mu_n,\mu)\ge \frac{d}{(d+q)2^q}  \frac{1}{n^{\frac qd}}\]
so that in high dimension, for the $\ell^\infty$ norm and in the worst case $q=1$ our estimate is off by a factor of approximately $4$ compared to a best approximation.

With the Euclidean norm, an easy lower bound in the case of the Lebesgue measure is obtained by observing that a mass at most
\[ \frac{\pi^{\frac{d}{2}}}{\Gamma(\frac{d}{2}+1)} R^d n\]
is at distance $R$ or less of one of the $n$ points (be they random or not). This leads, for \emph{any} measure $\tilde\mu_n$ supported on $n$ points, to
\[\wass_1(\tilde\mu_n,\mu) \ge n\int_0^{R_0} d \frac{\pi^{\frac{d}{2}}}{\Gamma(\frac{d}{2}+1)} R^d \dd R = n\frac{d\pi^{\frac{d}{2}}}{(d+1)\Gamma(\frac{d}{2}+1)} R_0^{d+1}\]
where $R_0$ is defined by $n\frac{\pi^{\frac{d}{2}}}{\Gamma(\frac{d}{2}+1)} R_0^d=1$. Finally,
\[W_1(\tilde\mu_n,\mu) \ge \underbrace{\frac{d\Gamma(\frac{d}{2}+1)^{\frac1d}}{(d+1)\sqrt{\pi}}}_{\underset{d\to\infty}{\sim} \sqrt{\frac{d}{2e\pi}}} \cdot \frac{1}{n^{\frac1d}}\]
and again our order of magnitude $C_d\asymp \sqrt{d}$ is the correct one.

The results of \cite{talagrand1992matching} show that, at least for the bipartite matching problem, this seemingly crude lower bounds are in fact attained asymptotically, taking renormalized limits as $n\to\infty$ and then $d\to\infty$. This indicates that our constant are not optimal, and it would be interesting to have a non-asymptotic bound with optimal asymptotic behavior.
\end{rema}

%%%%%%%%%%%%%%%%%%%%%%%%%%%%%%%%%%%%%%%%%%%%%%%%%%%%%%%%%%%%%%%
\subsection{Decomposition of H\"older functions}

The method to prove Theorem \ref{theo:Wq} consists in a multiscale decomposition of the functions $f\in\Hol^{q,\infty}_1$. In its spirit, it seems quite close to arguments of \cite{boissard2014mean}, \cite{dereich2013constructive} and \cite{fournier2015rate}; our interest is mostly in setting this multiscale analysis in a functional decomposition framework.

We fix a positive integer $J$ to be optimized later, representing the depth of the decomposition.
For each $j \in \{0,\dots,J\}$, set $\Lambda_j = \{j\} \times \{0,\dots, 2^{j}-1\}^d$ ; then define $\Lambda = \bigcup_{j=0}^J \Lambda_j$, acting as the set of indices for the decomposition.

For each $j \in \{0,\dots,J\}$, let $\{C_\lambda : \lambda\in\Lambda_j\}$ be the regular decomposition of $[0,1]^d$ into cubes of side-length $2^{-j}$; the boundary points are attributed in an arbitrary (measurable) manner, with the constraint that $\{C_\lambda : \lambda\in\Lambda_j\}$ is a partition of $[0,1]^d$ that refines the previous partition $\{C_\lambda : \lambda\in\Lambda_{j-1}\}$. Denote by $x_\lambda$ the center of the cube $C_\lambda$, and by $\psi_\lambda := \one_{C_\lambda}$ the characteristic function of $C_\lambda$ (so that for each $j$, $\sum_{\lambda\in\Lambda_j} \psi_\lambda = \one_{[0,1]^d}$).

\begin{lemm}\label{lemm:HolDec}
For all  function $f\in \Hol^{q,\infty}_1$ and all $J$, there exists coefficients $\alpha(\lambda)\in\mathbb{R}$ such that
\begin{equation}
f = \sum_{j=1}^J \sum_{\lambda\in\Lambda_j} \alpha(\lambda) \psi_\lambda + c + g
\label{eq:HolDec}
\end{equation} 
where $c$ is a constant and $g$ is a function $[0,1]^d\to \mathbb{R}$, such that
\begin{align*}
\lvert \alpha(\lambda)\rvert &\le 2^{-(j+1)q} \qquad \forall \lambda\in\Lambda_j\\
\lVert g\rVert_\infty &\le 2^{-(J+1)q}.
\end{align*}
\end{lemm}

\begin{proof}
Replacing $f$ with $f-c$ where $c=f(x_{0,0})$, we assume that $f$ vanishes at the center $x_{0,0}$ of $C_{0,0}=[0,1]^d$. Observe that $f\in\Hol^{q,\infty}_1$ then implies that $\lVert f\rVert_\infty \le 2^{-q}$ and $\lvert f(x_\lambda)\rvert \le 2^{-2q}$ for all $\lambda\in\Lambda_1$.

For $\lambda\in\Lambda_1$, we define $\alpha(\lambda) = f(x_\lambda)$ and set $f_1 = \sum_{\lambda\in\Lambda_1} \alpha(\lambda) \psi_\lambda$; 
we have $\lvert \alpha(\lambda)\rvert\le 2^{-2q}$, the function $f-f_1$ is $\Hol^{q,\infty}_1$ on $C_\lambda$ and vanishes at $x_\lambda$. Since $C_\lambda$ is a $\lVert\cdot\rVert_\infty$ ball of center $x_\lambda$ and radius $1/4$, it follows that $\lVert f-f_1\rVert_\infty \le 2^{-2q}$ on each $C_\lambda$, and thus on the whole of $[0,1]^d$. Moreover for all $\lambda\in\Lambda_2$ it holds $\lvert (f-f_1)(x_\lambda)\rvert\le 2^{-3q}$.
 
Similarly, we define $f_j:[0,1]^d\to\mathbb{R}$ recursively by setting $\alpha(\lambda)= (f-f_{j-1})(x_\lambda)$ for all $\lambda\in\Lambda_j$ and $f_j=f_{j-1} + \sum_{\lambda\in\Lambda_j} \alpha(\lambda) \psi_\lambda$. Then $\lvert\alpha(\lambda)\rvert \le 2^{-(j+1)q}$ for all $\lambda\in\Lambda_j$ and $\lVert f-f_J\rVert_\infty\le 2^{-(J+1)q}$.
\end{proof}

%%%%%%%%%%%%%%%%%%%%%%%%%%%%%%%%%%%%%%%%%%%%%%%%%%%%%%%%%%%%%%%%%%%%%%%%%%%%%
\subsection{Wasserstein distance estimation}

With the notation of Lemma \ref{lemm:HolDec}, for any $f\in\Hol^q_1$ we have:
\begin{align*}
\big\lvert \hat\mu_n(f) - \mu(f) \big\rvert
  &\le 2\lVert g\rVert_\infty + \sum_{j=1}^J \sum_{\lambda\in\Lambda_j} \lvert \alpha(\lambda)\rvert \lvert \hat\mu_n(\psi_\lambda) - \mu(\psi_\lambda) \rvert \\
  &\le 2^{1-(J+1)q} + \sum_{j=1}^J 2^{-(j+1)q} \sum_{\lambda\in\Lambda_j} \lvert \hat\mu_n(\psi_\lambda) - \mu(\psi_\lambda) \rvert
\end{align*}
where the last right-hand term does not depend on $f$ in any way. We can thus take a supremum and an expectation to obtain
\begin{align*}
\esp\big[ \wass_{q,\infty}(\hat\mu_n,\mu) \big]
  &\le 2^{1-(J+1)q} + \sum_{j=1}^J 2^{-(j+1)q} \sum_{\lambda\in\Lambda_j} \esp\big[ \lvert \hat\mu_n(\psi_\lambda) - \mu(\psi_\lambda) \rvert \big]
\end{align*}

\begin{rema}
This is the core of the decomposition method. Observe that we used no hypothesis on the $(X_k)$ yet; any stochastic process for which one can control $\esp[\lvert \hat\mu_n(\psi_\lambda) - \mu(\psi_\lambda) \rvert ]$ can be applied the method.
\end{rema}

Setting $p_\lambda = \mu(\psi_\lambda)$, the random variable $n\hat\mu_n(\psi_\lambda)$ is binomial of parameters $n$ and $p_\lambda$. A standard estimation of the mean absolute deviation yields
\begin{align*}
\esp\big[ \lvert n\hat\mu_n(\psi_\lambda) - n\mu(\psi_\lambda) \rvert \big] &\le \sqrt{n p_\lambda(1-p_\lambda)} \\
\sum_{\lambda\in \Lambda_j} \esp\big[ \lvert \hat\mu_n(\psi_\lambda) - \mu(\psi_\lambda) \rvert \big]
  &\le \frac{1}{\sqrt{n}}\sum_{\lambda\in \Lambda_j} \sqrt{p_\lambda}
\end{align*}
By concavity of the square-root function, we have
\begin{equation} 2^{-dj} \sum_{\lambda\in \Lambda_j} \sqrt{p_\lambda} \le \sqrt{2^{-dj} \sum_{\lambda\in \Lambda_j}  p_\lambda} = 2^{-\frac{dj}{2}}
\label{eq:concavity}
\end{equation}
and we deduce
\begin{align}
\sum_{\lambda\in \Lambda_j} \esp\big[ \lvert \hat\mu_n(\psi_\lambda) - \mu(\psi_\lambda) \rvert \big]
  &\le \frac{2^{\frac{dj}{2}}}{\sqrt{n}} \nonumber\\
\esp\big[ \wass_{q,\infty}(\hat\mu_n,\mu) \big]
  &\le 2^{1-(J+1)q} + \sum_{j=1}^J \frac{2^{j(\frac{d}{2}-q)-q}}{\sqrt{n}},
\label{eq:esp-sum}
\end{align}
leaving us with the simple task to optimize the choice of $J$.

%%%%%%%%%%%%%%%%%%%%%%%%%%%%%%%%%%%%%%%%%%%%%%%%%%%%%%%%%%%%%%%%%%%%%%%%%%%%%%%
\subsection{Optimization of the depth parameter}

We shall distinguish three cases: $d<2q$, $d=2q$ and $d>2q$. The first case is only possible for $d=1$, but we let it phrased that way because for some measures $\mu$ the dimension $d$ of the ambient space can be replaced by the  ``dimension'' of the measure itself, see Section \ref{sec:four-corners} for an example.

\subsubsection{Small dimension}
If $d<2q$, then the sum in \eqref{eq:esp-sum} is bounded independently of $J$ and we can let $J\to\infty$ to obtain:
\begin{align}
\esp\big[ \wass_{q,\infty}(\hat\mu_n,\mu) \big]
  &\le \frac{2^{-q}}{\sqrt{n}}\sum_{j=1}^\infty 2^{j(\frac{d}{2}-q)} \nonumber\\
  &\le \frac{2^{\frac{d}{2}-2q}}{1-2^{\frac{d}{2}-q}}\cdot\frac{1}{\sqrt{n}}
\label{eq:small-d}
\end{align}

In particular, for $d=1$, $q=1$:
\begin{equation}
\esp\big[ \wass_{1}(\hat\mu_n,\mu) \big] \le \frac{1}{2(\sqrt{2}-1)}\cdot\frac{1}{\sqrt{n}}
\end{equation}

\begin{rema}
For $\frac{d}{2}-q$ close to $0$, the constant in \eqref{eq:small-d} goes to infinity; in this regime, for moderate $n$ letting $J\to\infty$ is sub-optimal and one should optimize $J$ in \eqref{eq:esp-sum} as we shall do in the next cases.
\end{rema}

\subsubsection{Critical dimension}
If $d=2q$ (or in fact $d\le 2q$) we can rewrite \eqref{eq:esp-sum} as
\[\esp\big[ \wass_{q,\infty}(\hat\mu_n,\mu) \big]
  \le 2^{1-(J+1)q} + \frac{2^{-q} J}{\sqrt{n}}.\]
To optimize $J$, we formally differentiate the right-hand side with respect to $J$, equate to zero and solve for $J$. Reminding that $J$ is an integer, and keeping only the leading term (when $n\to\infty$) to simplify, this leads us to choose
\[ J = \Big\lfloor \frac{\log_2 n}{2q} \Big\rfloor \]
in particular implying $2^{1-(J+1)q} \le 2/\sqrt{n}$.
We deduce the claimed bound
\begin{equation}
\esp\big[ \wass_{q,\infty}(\hat\mu_n,\mu) \big]
  \le \Big(2 + \frac{\log_2(n)}{2^{q+1}q} \Big) \frac{1}{\sqrt{n}} \lesssim \frac{\log n}{\sqrt{n}}
\label{eq:critical-d}
\end{equation}
immediately implying  the bound of Theorem \ref{theomain:W1} for $d=2$ and $q=1$ (where a $\sqrt{2}$ comes from the comparison between the supremum and Euclidean norms):
\begin{equation}
\esp\big[ \wass_{1}(\hat\mu_n,\mu) \big]
  \le \frac{\log_2(n)+8}{\sqrt{8n}}
\end{equation}

%\begin{rema}
%It is an open question whether, for a general $\mu$, the order can be reduced to $\sqrt{\log(n)/n}$ in \eqref{eq:critical-d}, as it can be done in the case of the Lebesgue measure. However, this cannot be done for some singular measure of dimension $1$ on the square and $q=\frac12$, where the order $\log(n)/ \sqrt{n}$ is optimal, see Section \ref{sec:four-corners}.
%\end{rema}

\subsubsection{Large dimension}
If $d>2q$, equation \eqref{eq:esp-sum} becomes
\[ \esp\big[ \wass_{q,\infty}(\hat\mu_n,\mu) \big]
  \le 2^{1-(J+1)q} + \frac{2^{J(\frac{d}{2}-q)}-1}{1-2^{q-\frac{d}{2}}} \cdot \frac{1}{2^q\sqrt{n}} \le 2^{1-(J+1)q} + \frac{2^{J(\frac{d}{2}-q)}}{2^q(1-2^{q-\frac{d}{2}})} \cdot \frac{1}{\sqrt{n}} \]
Following the same optimization process as in the critical dimension case, we choose $J$ such that
\[ \frac12 n^{\frac1d} \Big(\frac{2q(1-2^{q-\frac{d}{2}})}{\frac{d}{2}-q} \Big)^{\frac2d} \le 2^J \le   n^{\frac1d} \Big(\frac{2q(1-2^{q-\frac{d}{2}})}{\frac{d}{2}-q} \Big)^{\frac2d}\]
leading to 
\begin{equation*}
\esp\big[ \wass_{q,\infty}(\hat\mu_n,\mu) \big]
  \le 2 \Big(\frac{\frac{d}{2}-q}{2q(1-2^{q-\frac{d}{2}})}\Big)^{\frac{2q}{d}} \Big( 1  + \frac{q}{2^q(\frac{d}{2}-q)} \Big) \frac{1}{n^{\frac{q}{d}}}
\end{equation*}

For $q=1$ and $d\ge 3$, it comes $\esp\big[ \wass_{1,\infty}(\hat\mu_n,\mu) \big] \le C'_d n^{-\frac1d}$ where
\begin{align*}
C'_d = 2\Big(\frac{\frac{d}{2}-1}{2-2^{2-\frac{d}{2}}}\Big)^{\frac{2}{d}} \Big( 1  + \frac{1}{d-2} \Big) \frac{1}{n^{\frac{1}{d}}}
\end{align*}
We have notably $C'_4=3$. Relaxing our bound for $d\ge 4$ to
\[C'_d\le 2\Big(\frac{d}{4}\Big)^{\frac{2}{d}} \Big( 1  + \frac{1}{d-2} \Big)\]
it is more easily seen that it is decreasing (and still takes the value $3$ at $d=4$). We also see that we can take $C'_d\to 2$ as $d\to\infty$. The last part of Theorem \ref{theomain:W1} follows with $C_d = \sqrt{d} C'_d$, and a numerical computation shows $C_3\le 6.3$.

%%%%%%%%%%%%%%%%%%%%%%%%%%%%%%%%%%%%%%%%%%%%%%%%%%%%%%%%%%%%%%%%%%%%%%%%%%%%%%
\subsection{The four-corners Cantor measure}\label{sec:four-corners}

We conclude this section with an example showing that the critical case order $\log n/\sqrt{n}$ is sharp if one generalizes its scope.

The \emph{four-corner} Cantor set $K$ is the compact subset of the plane defined as the attractor of the Iterated Function System $(T_1,T_2,T_3,T_4)$ where $T_i$ are homotheties of ratio $1/4$ centered at $(0,0)$, $(0,1)$, $(1,1)$ and $(1,0)$ (see figure \ref{fig:4corners}). It has a natural measure $\mu_K$, which can be defined as the fixed point of the map
\begin{align*}
\mathcal{T} \colon \proba([0,1]^2) &\to \proba([0,1]^2) \\
\nu &\mapsto \frac14 (T_1)_*\nu + \frac14 (T_2)_*\nu + \frac14 (T_3)_*\nu + \frac14 (T_4)_*\nu
\end{align*}
($\mathcal{T}$ is contracting in the complete metric $\wass_1$, so that it has a unique fixed point). The measure $\mu_K$ can also be described as follows. In the  $4$-adic decomposition of the square, at depth $j>0$ there are $16^j$ squares, among which $4^j$ intersect $K$ in their interior; $\mu_K$ gives each of these squares a mass $1/4^j$.

\begin{figure}
\centering
\includegraphics[scale=.3]{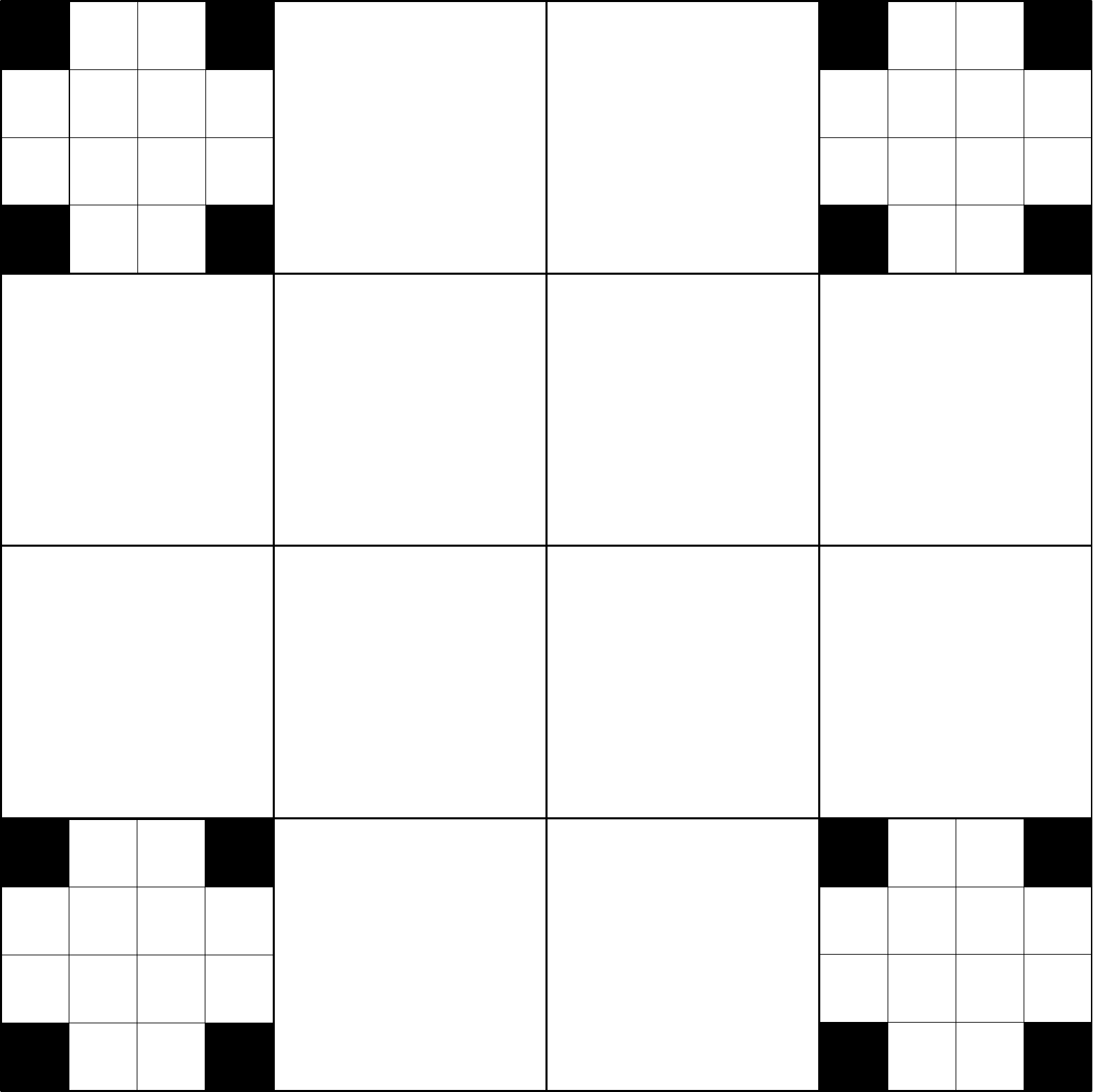}
\caption{Second stage of the construction of the four-corners Cantor set (contained in the filled black area).}
\label{fig:4corners}
\end{figure}

$K$ has Hausdorff dimension $1$ (and positive, finite $1$-dimensional Hausdorff measure), and one should expect $\mu_K$ to have dimension $d=1$ in any reasonable sense of the term. It is thus interesting to have a look at $\wass_{q}(\hat\mu_n,\mu_K)$ in the critical case $q=1/2$.

\begin{prop}
If $(X_k)_{k\ge0}$ are i.i.d. of law $\mu_K$, then
\[\esp\big[ \wass_{\frac12}(\hat\mu_n,\mu_K) \big] \asymp \frac{\log n}{\sqrt{n}}.\]
\end{prop}

\begin{proof}
The proof of the upper bound follows the proof of Theorem \ref{theo:Wq}, using a $4$-adic decomposition and discarding all $\lambda$ such that $C_\lambda$ does not intersects $K$ in its interior. This replaces $d$ by $1$ as there are $4^j$ relevant squares of size $4^{-j}$ (indeed the only place where $d$ is used is in \eqref{eq:concavity}, only through the number of dyadic squares to be considered), so that with $q=1/2$ we end up in the critical case.

To prove the lower bound, we first record the proportions $p_1,p_2,p_3,p_4$ of the random points $X_k$ lying in each of the four relevant depth-one squares (of side-length $1/4$). For large $n$, each $p_i$ is close to $1/4$ with typical fluctuations of the order of $1/\sqrt{n}$. The discrepancy of mass in each of these squares compared to the mass $1/4$ given to each of them by $\mu_K$ induces a cost of at least $1/\sqrt{2n}$, since the distance between depth-one squares is at least $1/2$ and $q=1/2$.
The same reasoning applies at depth two inside each depth-one square, but with $np_i \simeq n/4$ points, thus fluctuations are of the order of $1/\sqrt{n/4}=2/\sqrt{n}$, inducing a total cost of the order of $1/\sqrt{2n}$ (distances are now $1/4\times 1/2$, and a square root is taken since $q=1/2$). The fact that the number of points is $n p_i$ rather than precisely $n/4$ is not an issue, an uneven distribution improving the bound.

At each depth $j$ up to $\log_4 n$, there is a typical induced cost of the order of $1/\sqrt{n}$ from the uneven distribution of points among the $4$ subsquares of each depth $j$ square, yielding the desired bound of the order of $\log n/\sqrt{n}$.
\end{proof}

%%%%%%%%%%%%%%%%%%%%%%%%%%%%%%%%%%%%%%%%%%%%%%%%%%%%%%%%%%%%%%%
%%%%%%%%%%%%%%%%%%%%%%%%%%%%%%%%%%%%%%%%%%%%%%%%%%%%%%%%%%%%%%%
\section{Wavelet decomposition and convergence against regular test functions}
\label{sec:main}

%%%%%%%%%%%%%%%%%%%%%%%%%%%%%%%%%%%%%%%%%%%%%%%%%%%%%%%%%%%%%%%
\subsection{Wavelet decomposition}\label{sec:wavelets}

Let us give a short account of the results about wavelets we will use (see e.g. Meyer's book \cite{meyer1992wavelets} for proofs and references).

It will be convenient to use wavelets of compact support with arbitrary regularity $\Ck{r}$, whose construction is due to Daubechies \cite{daubechies1988orthonormal}. The construction yields \emph{compactly supported} functions $\phi,\psi^\epsilon:\mathbb{R}^d\to\mathbb{R}$ where $\epsilon$ takes any of $2^d-1$ values ($\epsilon\in E:=\{0,1\}^d\setminus\{(0,0,\dots,0)\}$), with particular properties of which only those we will use will be described. 

One defines from these ``father and mother'' wavelets a larger family of \emph{wavelets} by
\begin{align}
\phi_\tau(x) &=\phi(x-\tau),  && (\tau \in\mathbb{Z}^d) \nonumber\\
\psi_\lambda(x) &= 2^{\frac{dj}{2}}\psi^\epsilon(2^j x - \tau), && (\lambda=(j,\tau,\epsilon) \in \Lambda = \mathbb{Z}\times \mathbb{Z}^d\times E);
\label{eq:psi}
\end{align}
one important property of the construction is that the union of $(\phi_\tau)_{\tau\in\mathbb{Z}^d}$ and $(\psi_\lambda)_{\lambda\in\Lambda}$ form an \emph{orthonormal basis} of $L^2(\mathbb{R}^d)$. For 
$f\in L^2(\mathbb{R}^d)$ we can thus write
\[f = \sum_{\tau \in \mathbb{Z}^d} \langle f,\phi_\tau\rangle \phi_\tau + \sum_{j=0}^\infty \sum_{\lambda\in\Lambda_j} \langle f,\psi_\lambda\rangle \psi_\lambda \]
where $\Lambda_j = \{j\}\times \mathbb{Z}^d\times E$ and $\langle\cdot,\cdot\rangle$ denotes the $L^2$ scalar product (with respect to Lebesgue measure).

One stunning property is that many functional spaces can be \emph{characterized} in term of the wavelet coefficients
$\alpha(\lambda)=\langle f,\psi_\lambda\rangle$ and $\beta(\tau)=\langle f,\phi_\tau\rangle$. We shall only use upper bounds on the $\alpha(\lambda)$ and $\beta(\tau)$ in a specific case.

The H\"older space $\Ck{s}$ is defined as the space of $k$ times continuously differentiable with $\gamma$-H\"older partial derivatives of order $k$, with $k$ a non-negative integer, $\gamma\in(0,1]$ and $k+\gamma=s$ (e.g. $\Ck{1}$ is the space of Lipschitz functions, $\Ck{3/2}$ the space of once continuously differentiable functions with $1/2$-H\"older first-order partial derivatives, $\Ck{5}$ is the space of four-times continuously differentiable functions with  Lipschitz fourth-order partial derivatives, etc.). Note that ``$1$-H\"older'',  meaning ``Lipschitz'', could be slightly enlarged to ``Zygmund'' (and should, if one is interested in two-sided bounds), but we need not enter this subtlety here. 

The space $\Ck{s}$ is endowed with the norm
\[\lVert f\rVert_{\Ck{s}} = \max_{j\in\{0,\dots,k\}} \max_{\omega\in \{1,\dots,d\}^{j} } \Big\lVert \frac{\partial^j f}{\partial x_{\omega_1}\cdots\partial x_{\omega_j}}\Big\lVert_\star \]
where the decomposition $s=k+\gamma$ is defined as above and $\lVert\cdot\rVert_\star$ is the uniform norm if $j<k$ and is the $\gamma$-H\"older constant if $j=k$. We denote by $\Cku{s}$ the set of functions with  $\Ck{s}$ norm at most $1$.

If the regularity of the wavelets is larger than the regularity of the considered H\"older space ($r>s$) then
\begin{alignat*}{2}
\lvert \beta(\tau) \rvert &\le C_{d,s}\lVert f\rVert_\infty &&\quad\forall \tau\in\mathbb{Z}^d \\
\lvert \alpha(\lambda)\rvert &\le C_{d,s} \lVert f \rVert_{\Ck{s}} 2^{-\frac{dj}{2}} 2^{-js} &&\quad \forall \lambda\in\Lambda_j,
\end{alignat*}
where the constant $C_{d,s}$ depends implicitely on the choice of father and mother wavelets $\phi$ and $\psi^\epsilon$; but we can fix for each $s$ such a choice with suitable regularity, e.g. $r=s+1$ and the constants then truly depends only on $d$ and $s$. The $\Ck{s}$ norm in the $\alpha(\lambda)$ coefficient could be relaxed to the ``regularity part'' of the norm but we do not use this.

Note that the explicit computation of these constants would in particular need a very fine analysis of the chosen wavelet construction, and I do not know whether such a task has been conducted.

%%%%%%%%%%%%%%%%%%%%%%%%%%%%%%%%%%%%%%%%%%%%%%%%%%%%%%%%%%%%%%%
\subsection{Decomposition of regular functions}

Let us now use wavelet decomposition to prove good convergence properties for the empirical measure against smooth enough test functions; the strategy is similar to the one used in Section \ref{sec:W1}.  We assume here that $(X_k)_{k\ge0}$ is a sequence of i.i.d. random variables whose law $\mu$ is supported on a bounded set $\Omega\subset \mathbb{R}^d$ (e.g. $\Omega=[0,1]^d$); note that $\Cku{s}=\Cku{s}(\mathbb{R}^d)$ makes no reference to $\Omega$. We consider a fixed family of wavelet of regularity $r>s$ as in Section \ref{sec:wavelets}; all constants $C$ below implicitly depend on $d$, $s$ and $\Omega$ (only through its diameter). 

Since the wavelets have compact support, there exist some constant $C$ such that for each $j$:
\begin{itemize}
\item for each point $x\in[0,1]^d$, there are at most $C$ different $\lambda$ corresponding to a $\psi_\lambda$ that does not vanish at $x$; the set of those $\lambda$ is denoted by $\Lambda_j(x)\subset\Lambda_j$,
\item the union $\Lambda_j(\Omega) := \bigcup_{x\in\Omega} \lambda_j(x)$
has at most $C 2^{dj}$ elements.
\end{itemize}
We denote by $Z$ the set of parameters $\tau\in\mathbb{Z}^d$ corresponding to a $\phi_\tau$ whose support intersects $\Omega$ (observe that $Z$ is finite).

We fix a function $f\in\Cku{s}$ and decompose it in our wavelet basis:
\[f = \sum_{\tau \in \mathbb{Z}^d} \beta(\tau) \phi_\tau + \sum_{j=0}^\infty \sum_{\lambda\in\Lambda_j} \alpha(\lambda) \psi_\lambda \]
with
\begin{alignat*}{2}
\lvert \beta(\tau) \rvert &\lesssim 1 &&\quad \forall \tau\in\mathbb{Z}^d\\
\lvert \alpha(\lambda)\rvert &\lesssim 2^{-\frac{dj}{2}} 2^{-js} &&\quad \forall \lambda\in\Lambda_j.
\end{alignat*}

Cutting the second term of the decomposition to some depth $J$ we get:
\[f = \sum_{\tau \in Z} \beta(\tau) \phi_\tau + \sum_{j=0}^J \sum_{\lambda\in\Lambda_j} \alpha(\lambda) \psi_\lambda + g\]
where 
\[ g = \sum_{\tau \notin Z} \beta(\tau) \phi_\tau + \sum_{j>J} \sum_{\lambda\in\Lambda_j} \alpha(\lambda) \psi_\lambda.\]
Using the bound on the $\alpha$ coefficients and the formula \eqref{eq:psi} for $\psi_\lambda$, we get:
\[\lVert g\one_{\Omega}\rVert_\infty \lesssim 2^{-s J}\]
and it follows:
\[
\lvert \hat\mu_n(f) - \mu(f) \rvert \lesssim 2^{-Js} + \sum_{\tau\in Z} \lvert \hat\mu_n(\phi_\tau)-\mu(\phi_\tau)\rvert + \sum_{j=0}^J \sum_{\lambda\in\Lambda_j(\Omega)} 2^{-(\frac{d}{2}+s)j} \lvert \hat\mu_n(\psi_\lambda) - \mu(\psi_\lambda) \rvert
\]
where the right-hand side does not depend on $f$. Taking a supremum and an expectation, it then comes:
\begin{equation}
\esp\big[ \lVert \hat\mu_n - \mu\rVert_{\Cku{s}} \big] \lesssim 
2^{-sJ} + \sum_{\tau\in Z} \esp\big[ \lvert \hat\mu_n(\phi_\tau)-\mu(\phi_\tau)\rvert\big] + \sum_{j=0}^J \sum_{\lambda\in\Lambda_j(\Omega)} 2^{-(\frac{d}{2}+s)j} \esp\big[ \lvert \hat\mu_n(\psi_\lambda) - \mu(\psi_\lambda) \rvert\big]
\label{eq:wavelet1}
\end{equation}
and to conclude, we simply need to estimate the last two terms above.

%%%%%%%%%%%%%%%%%%%%%%%%%%%%%%%%%%%%%%%%%%%%%%%%%%%%%%%%%%%%%%%
\subsection{Convergence for basis elements}

\begin{lemm}\label{lemm:varsum}
We have 
\begin{align*}
\sum_{\tau\in Z} \esp\big[ \lvert \hat\mu_n(\phi_\tau)-\mu(\phi_\tau)\rvert\big] &\lesssim \frac{1}{\sqrt{n}} \\
\text{and}\qquad \sum_{\lambda\in\Lambda_j(\Omega)} \esp\big[ \lvert \hat\mu_n(\psi_\lambda) - \mu(\psi_\lambda) \rvert\big] &\lesssim \frac{2^{dj}}{\sqrt{n}}
\end{align*}
\end{lemm}

\begin{proof}
For each $\tau\in Z$, the random variable $\hat \mu_n(\phi_\tau)$ is the average of $n$ independent identically distributed, bounded random variables of expectation $\mu(\phi_\tau)$, so that
$\esp\big[ \lvert \hat \mu_n(\phi_\tau)-\mu(\phi_\tau)\rvert\big] \le C/\sqrt{n}$. Since $Z$ is finite, the first claim is proved.

To prove the second claim, we cannot argue in the exact same way because $\psi_\lambda$ depends on $j$. To ease notation we introduce $\bar\psi_\lambda :=  2^{-\frac{dj}{2}} \psi_\lambda$ and $Y_\lambda :=  \hat\mu_n(\bar\psi_\lambda)-\mu(\bar\psi_\lambda)$, and recall that $\bar\psi_\lambda$ is bounded independently of $j$. Also, a bounded number of different $\bar\psi_\lambda$ ($\lambda\in\Lambda_j$) are non-zero at any point $x\in\Omega$; we denote by $p_\lambda$ the mass given by $\mu$ to the support of $\psi_\lambda$ and observe that $Y_\lambda$ is the average of $n$ i.i.d. centered random variables of variance less than $Cp_\lambda + \mu(\bar\psi_\lambda)^2$.
We have
\[\Var(Y_\lambda) \le \frac1n\big(C p_\lambda + \mu(\bar\psi_\lambda)^2\big) \qquad \sum_{\lambda\in\Lambda_j(\Omega)} p_\lambda \lesssim 1 \qquad \sum_{\lambda\in\Lambda_j(\Omega)} \mu(\bar\psi_\lambda) \lesssim 1 \]
so that
\begin{align*}
\sum_{\lambda\in\Lambda_j(\Omega)} \Var(Y_k)
  &\le \frac1n \Big( C \sum_{\lambda\in\Lambda_j(\Omega)} p_\lambda + \big(\sum_{\lambda\in\Lambda_j(\Omega)} \mu(\bar\psi_\lambda)\big)^2 \Big)\\
  &\lesssim \frac1n.
\end{align*}
Now it comes
\begin{align*}
\sum_{\lambda\in\Lambda_j(\Omega)} \esp\big[ \lvert \hat\mu_n(\psi_\lambda) - \mu(\psi_\lambda) \rvert\big]
  &= 2^{\frac{dj}{2}} \sum_{\lambda\in\Lambda_j(\Omega)} \esp\big[ \lvert Y_\lambda \rvert\big] \\
  &\le 2^{\frac{dj}{2}}\sum_{\lambda\in\Lambda_j(\Omega)} \sqrt{\esp\big[ Y_\lambda^2 \big]} \\
  &\le 2^{\frac{dj}{2}}\sqrt{\lvert \Lambda_j(\Omega)\rvert} \sqrt{\sum_{\lambda\in\Lambda_j(\Omega)} \Var(Y_\lambda)}\\
  &\lesssim \frac{2^{dj}}{\sqrt{n}}
\end{align*}
\end{proof}

\begin{rema}
Lemma \ref{lemm:varsum} is the only place where we use that the $(X_k)_{k\in\mathbb{N}}$ are i.i.d. The method can therefore be applied to any stochastic process satisfying the conclusion of Lemma \ref{lemm:varsum}.
\end{rema}

%%%%%%%%%%%%%%%%%%%%%%%%%%%%%%%%%%%%%%%%%%%%%%%%%%%%%%%%%%%%%%%
\subsection{Conclusion of the proof}

Plugin Lemma \ref{lemm:varsum} into \eqref{eq:wavelet1} yields
\begin{align*}
\esp\big[ \lVert \hat\mu_n - \mu\rVert_{\Cku{s}} \big] 
%  &\lesssim 
% 2^{-Js} + \frac{1}{\sqrt{n}} + \sum_{j=0}^J  \frac{2^{(\frac{d}{2}-s)j}}{\sqrt{n}} \\
 &\lesssim 2^{-Js} + \frac{1}{\sqrt{n}} \sum_{j=0}^J \big(2^{\frac{d}{2}-s}\big)^j
\end{align*}
and we get the same trichotomy as before. If $s > d/2$, then we can let $J\to\infty$ to obtain 
\[\esp\big[ \lVert \hat\mu_n - \mu\rVert_{\Cku{s}} \big] \le 
\frac{C}{\sqrt{n}},\]
if $s=d/2$ we can take $J$ such that $2^{-Js}\simeq 1/\sqrt{n}$ and get
\[\esp\big[ \lVert \hat\mu_n - \mu\rVert_{\Cku{s}} \big] \le C
\frac{\log n}{\sqrt{n}},\]
and if $s< d/2$ we can choose $J$ such that $2^J\simeq n^{\frac1d}$ to get
\[\esp\big[ \lVert \hat\mu_n - \mu\rVert_{\Cku{s}} \big] \le 
\frac{C}{n^{s/d}},\]
ending the proof of Theorem \ref{theomain:reg}.

%%%%%%%%%%%%%%%%%%%%%%%%%%%%%%%%%%%%%%%%%%%%%%%%%%%%%%%%%%%%%%%
%%%%%%%%%%%%%%%%%%%%%%%%%%%%%%%%%%%%%%%%%%%%%%%%%%%%%%%%%%%%%%%
\section{Markov chains}\label{sec:MC}

In this section we assume  $(X_k)_{k\ge 0}$ is a Markov chain on a bounded domain; since we will use Fourier series, it will make things simpler to embed this domain into a torus, so we assume $\Omega\subset \mathbb{T}^d = \mathbb{R}^d/\mathbb{Z}^d$ (we do not lose generality in doing so, as scaling down $\Omega$ makes it possible to make the embedding isometric). We still denote by $\lVert x-y\rVert$ the distance between two points induced by the Euclidean norm.

Our main assumption is that the iterated transition kernel of $(X_k)_{k\ge0}$, defined by
\[ m_x(A) = \pr(X_{k+1}\in A \mid X_k=x)  \qquad m_x^t(A) = \pr(X_{k+t}\in A \mid X_k=x)\]
 is exponentially contracting in $\wass_1$, i.e. there are constants $D\ge 1$ and $\theta\in(0,1)$ such that
\begin{equation}
\wass_1(m_x^t,m_y^t) \le D\theta^t \lVert x-y\rVert.
\label{eq:contraction}
\end{equation}

Let us denote by $\op{L}$  the averaging operator, i.e.
\[\op{L} f (x) = \int f(y) \dd m_x(y)\]
and by $\op{L}^*$ its dual acting on probability measure, i.e. $\op{L}^*\nu$ is the law of $X_{k+1}$ conditioned on $X_k$ having law $\nu$. The linearity of $\wass_1$ enables one to rewrite \eqref{eq:contraction} as
\begin{equation}
\wass_1(\op{L}^{*t}\nu_0,\op{L}^{*t}\nu_1) \le D\theta^t \wass_1(\nu_0,\nu_1) 
\label{eq:contraction2}
\end{equation}
so that there is a unique stationary measure $\mu$, and the law of $X_k$ converges exponentially fast (in $\wass_1$) to $\mu$, whatever the law of $X_0$ is.

We shall prove Theorem \ref{theomain:Markov}, which we restate for convenience.
\begin{theo}\label{theo:Markov-precise}
For some constant $C=C(\Omega,d,D,s)$ and all large enough $n$, letting $\bar n=(1-\theta)n$, we have
\begin{equation}
\esp\big[ \lVert \hat\mu_n - \mu \rVert_{\Cku{s}} \big] \le C \begin{dcases*}
\frac{(\log \bar n)^{\frac{d}{2s+1}}}{\sqrt{\bar n}} & when $s > d/2$\\[2\jot]
\frac{\log \bar n}{\sqrt{\bar n}} & when $s=d/2$ \\[2\jot]
\frac{(\log \bar n)^{d-2s+\frac sd}}{\bar n^{\frac sd}} & when $s < d/2$
\end{dcases*}
\label{eq:theo-Markov}
\end{equation}
\end{theo}

Following the decomposition method, we shall find a suitable decomposition basis for any $f\in\Cku{s}$, seeking for a compromise between precision of a truncated decomposition and number of basis elements. Here using wavelets seems inefficient, as we do not have a precise enough analogue of Lemma \ref{lemm:varsum}, which uses independence to take advantage of the localization property of wavelets; without this, the number and size of the $\psi_\lambda$ are overwhelming. We shall use Fourier series instead, as they will be more easily controlled under our assumptions. For simplicity we consider complex-valued functions here, and denote the Fourier basis by
$e_k(x) := e^{2i\pi k\cdot x}$
where $k\in\mathbb{Z}^d$ and the dot $\cdot$ denotes the canonical inner product.

The key is thus to control $\lvert \hat\mu_n(e_k) - \mu(e_k) \rvert$; our hypothesis may seem perfectly suited to this since $e_k$ is Lipschitz, but its Lipschitz constant grows too rapidly with $k$ for a direct approach to be efficient. We shall combine the following two observations (the first of which is pretty trivial, the second of which is folklore).
\begin{lemm}\label{lemm:Hol-ek}
For all $\alpha\in(0,1)$, we have the following control of $e_k$'s $\alpha$-H\"older constant:
\[\Hol_\alpha(e_k) \lesssim \lvert k\rvert_\infty^\alpha\]
where $\lvert k\rvert_\infty = \max \big\{k_i : i\in\{1,\dots,d\} \big\}$.
\end{lemm}

\begin{proof}
We have $\Lip(e_k) \le 2\pi \sqrt{d} \lvert k\rvert_\infty$ and $\lVert e_k\rVert_\infty\le 1$ so that for all $x\neq y\in\mathbb{T}^d$:
\[
\frac{\lvert e_k(x)-e_k(y) \rvert}{\lVert x-y\rVert^\alpha}
  \le \min\Big( \frac{2}{\lVert x-y\rVert^\alpha}, 2\pi\sqrt{d}\lvert k\rvert_\infty \lVert x-y\rVert^{1-\alpha}  \Big)
  \le 2\pi^\alpha d^{\frac\alpha2}\lvert k\rvert_\infty^\alpha
\]
\end{proof}

\begin{lemm}\label{lemm:correlations}
For all $\alpha\in(0,1]$, denoting by $\wass_\alpha$ the $\alpha$-Wasserstein metric (i.e. the $1$-Wasserstein metric associated with the modified distance $\lVert \cdot\rVert^\alpha$), we have
\begin{equation}
\wass_\alpha(\op{L}_0^{*t}\nu_0,\op{L}_0^{*t}\nu_1) \le D^\alpha\theta^{\alpha t} \wass_\alpha(\nu_0,\nu_1) 
\label{eq:contraction-alpha}
\end{equation}
As a consequence, for all $\alpha$-H\"older functions $f:\Omega\to \mathbb{C}$ and all $\ell,m\in\mathbb{N}$ it holds
\begin{align*}
\big\lvert \esp[f(X_\ell)]-\mu(f) \big\rvert &\lesssim \Hol_\alpha(f) \, \theta^{\alpha \ell} \\
\big\lvert \esp[f(X_m)f(X_\ell)] - \esp[f(X_m)] \esp[f(X_\ell)] \big\lvert   
  &\lesssim \Hol_\alpha(f)^2 \, \theta^{\alpha \lvert m-\ell \rvert}
\end{align*}
where the implied constants depends only on $\Omega$ and the constant $C$ in \eqref{eq:contraction}.
\end{lemm}

\begin{proof}
By linearity we only have to check \eqref{eq:contraction-alpha} when $\nu_0=\delta_x$ and $\nu_1=\delta_y$ for some $x,y\in\Omega$, and by concavity
\[ \wass_\alpha(\op{L}^{*t}\delta_x,\op{L}^{*t}\delta_y)\le \big( \wass_1(\op{L}^{*t}\delta_x,\op{L}^{*t}\delta_y) \big)^\alpha \le D^\alpha\theta^{\alpha t} \lVert x-y\rVert^\alpha = D^\alpha \theta^{\alpha t}\wass_\alpha(\delta_x,\delta_y).\]

To prove convergence toward the average and decay of correlation, we first use the contraction and that $\mu$ is the stationary measure to get
\begin{align*}
\big\lvert \op{L}^t f(x) - \mu(f)\big\rvert &= \Big\lvert \int \op{L}^tf \dd\delta_x -\int f \dd\mu \Big\rvert \\
  &= \Big\lvert \int f \dd\big( \op{L}^{*t}\delta_x\big) -\int f \dd \big(\op{L}^{*t} \mu \big) \Big\rvert \\
  &\le \Hol_\alpha(f) \wass_\alpha(\op{L}^{*t} \delta_x,\op{L}^{*t}\mu) \\
  &\le  \Hol_\alpha(f) \, D^\alpha\theta^{\alpha t} \wass_\alpha(\delta_x,\mu) \\
\big\lvert \op{L}^t f(x) - \mu(f)\big\rvert &\lesssim  \Hol_\alpha(f) \, \theta^{\alpha t}.
\end{align*} 

Assuming without lost of generality $\mu(f)=0$ we have $\lVert f\rVert _\infty\lesssim \Hol_\alpha(f)$ ($\mu(f)=0$ implies that $f$ takes both non-positive and non-negative values, and $\Omega$ is bounded). 
Assume further $m\ge \ell$ and write $m=\ell+t$.
Combining all previous observations we get:
\begin{align*}
\lVert \op{L}^t f\rVert_\infty &\lesssim \Hol_\alpha(f) \, \theta^{\alpha t},\\
\big\lvert \esp[f(X_m)] \big\lvert &= \big\lvert \esp\big[\op{L}^m f(X_0) \big] \big\rvert \\
  &\lesssim \Hol_\alpha(f) \, \theta^{\alpha m},\\
\big\lvert \esp[f(X_\ell)] \big\lvert &\lesssim \Hol_\alpha(f) \, \theta^{\alpha \ell},\\
\big\lvert\esp[f(X_m)f(X_\ell)] \big\rvert &= \big\lvert \esp\big[\op{L}^t f(X_\ell) \, f(X_\ell) \big] \big\rvert \\
  &\lesssim \lVert \op{L}^t f\rVert_\infty \esp[\lvert f(X_\ell)\rvert] \\
  &\lesssim \Hol_\alpha(f)^2 \theta^{\alpha t}
\end{align*}
and the conclusion follows.
\end{proof}

We deduce the following from these two Lemmas.
\begin{coro}
For all $k,\alpha$ and all $n\ge 1/(1-\theta^\alpha)$ it holds
\[\esp\big[\lvert \hat\mu_n(e_k)-\mu(e_k)\rvert^2\big] \lesssim \frac{\lvert k\rvert_\infty^{2\alpha}}{(1-\theta^\alpha)n} \]
\end{coro}

\begin{proof}
We have:
\begin{align*}
\esp\big[\lvert \hat\mu_n(e_k)-\mu(e_k)\rvert^2\big]
  &= \esp\Big[ \Big(\frac1n\sum_{\ell=1}^n e_k(X_\ell)-\mu(e_k) \Big)^2 \Big] \\
  &=\frac{1}{n^2}\sum_{1\le \ell,m\le n} \esp[ e_k(X_\ell) e_k(X_m) ] - \frac2n \sum_{\ell=1}^n \esp[ e_k(X_\ell) ] \mu(e_k) + \mu(e_k)^2 \\
  &\le \frac{1}{n^2} \Big( \sum_{1\le \ell,m\le n} \esp[ e_k(X_\ell)] \esp[ e_k(X_m) ] + C\Hol_\alpha(e_k)^2 \, \theta^{\alpha \lvert \ell-m\rvert} \Big) \\
   &\qquad\qquad - \frac2n \sum_{\ell=1}^n \esp[ e_k(X_\ell) ] \mu(e_k) + \mu(e_k)^2  \\
   &\le \frac{C\Hol_\alpha(e_k)^2}{n^2} \sum_{1\le \ell,m\le n}  \, \theta^{\alpha \lvert \ell-m\rvert} + \frac{1}{n^2} \Big(\sum_{\ell=1}^n \big(\esp[ e_k(X_\ell) ] -\mu(e_k)\big)  \Big)^2 \\
      &\lesssim \frac{\Hol_\alpha(e_k)^2}{n^2}\cdot \sum_{\ell=1}^n 2\sum_{t=0}^\infty \theta^{\alpha t} + \frac{\Hol_\alpha(e_k)^2}{n^2}\Big(\sum_{\ell=1}^n \theta^{\alpha\ell} \Big)^2 \\
   &\lesssim \frac{\Hol_\alpha(e_k)^2}{n^2}\cdot \frac{n}{1-\theta^\alpha} + \frac{\Hol_\alpha(e_k)^2}{n^2(1-\theta^{\alpha})^2} \\
   &\lesssim \frac{\lvert k\rvert_\infty^{2\alpha}}{(1-\theta^\alpha)n}
\end{align*}
whenever $n\ge 1/(1-\theta^\alpha)$.
\end{proof}

Fix some threshold $J\ge 3$ and some exponent $\alpha\in(0,1]$, to be determined explicitly later on.

Let $f:\mathbb{T}^d \to \mathbb{R}$ be in $\Cku{s}$.
From the multidimensional version of Jackson's theorem \cite{schultz1969multivariate}, we know that there is a trigonometric polynomial $T_J(f)$ which is a linear combination of the $e_k$ for $\lvert k\rvert_\infty\le J$, such that
\[\lVert f-T_J(f)\rVert_\infty \lesssim \frac{1}{J^s}\]
We have no clear control on the coefficient of 
this optimal trigonometric polynomial, which need not be the Fourier coefficients of $f$.
But it is also known that the Fourier series of $f$ is within a factor $\simeq \lVert f\rVert_\infty (\log J)^d$ of the best approximation (see \cite{mason1980near-best} for an optimal constant), so that denoting by $F_J(f) := \sum_{\lvert k\rvert_\infty \le J} \hat f_k e_k$ the $J$-truncation of the Fourier series of $f$, we get
\[\lVert f-F_J(f)\rVert_\infty \lesssim \frac{(\log J)^d}{J^s}.\]

We can assume $\hat f_0=0$ by translating $f$, and what precedes yields:
\begin{align}
\lvert \hat\mu_n(f)-\mu(f)\rvert 
  &\le \lvert \hat\mu_n(f)-\hat\mu_n(F_J(f))\rvert + \lvert \hat\mu_n(F_J(f)) -\mu(F_J(f))\rvert + \lvert \mu(F_J(f))-\mu(f)\rvert \nonumber\\
  &\le  2\lVert f-F_J(f)\rVert_\infty + \sum_{0<\lvert k\rvert_\infty\le J} \lvert \hat f_k\rvert \lvert \hat\mu_n(e_k) - \mu(e_k)\rvert \label{eq:Markov-line2}\\
  &\lesssim \frac{(\log J)^d}{J^s}  + \Big( \sum_{0<\lvert k\rvert_\infty\le J} \lvert \hat f_k\rvert^2 \lvert k\rvert_\infty^{2s} \Big)^{\frac12} \Bigg(\sum_{0<\lvert k\rvert_\infty\le J} \frac{\lvert \hat\mu_n(e_k) - \mu(e_k)\rvert^2}{\lvert k\rvert_\infty^{2s}} \Bigg)^{\frac12} \nonumber\\
  &\lesssim \frac{(\log J)^d}{J^s} + \lVert f\rVert_{H^s} \Bigg( \sum_{0<\lvert k\rvert_\infty\le J} \frac{\lvert \hat\mu_n(e_k) - \mu(e_k)\rvert^2}{\lvert k\rvert_\infty^{2s}}\Bigg)^{\frac12} \nonumber\\
\lvert \hat\mu_n(f)-\mu(f)\rvert
    &\lesssim \frac{(\log J)^d}{J^s} + \Bigg( \sum_{0<\lvert k\rvert_\infty\le J} \frac{\lvert \hat\mu_n(e_k) - \mu(e_k)\rvert^2}{\lvert k\rvert_\infty^{2s}}\Bigg)^{\frac12} \label{eq:Markov-main}
\end{align}
Where the right-hand side does not depend on $f$ in any way (note that $\lVert \cdot\rVert_{H^s}$ is the Sobolev norm, controlled by the $\Ck{s}$ norm).
\begin{rema}
At line \eqref{eq:Markov-line2}, one could be tempted to bound directl $\lvert \hat f_k\rvert$ instead of using the Cauchy-Schwarz inequality, in order to make better use of our assumption on $f$. This would be effective if $\lvert \hat\mu_n(e_k)-\mu(e_k)\rvert$ were of the order of $1/n$, but it is actually of the order of $1/\sqrt{n}$, ultimately leading to a weaker bound than the one we aim for.
\end{rema}

Taking a supremum and an expectation in \eqref{eq:Markov-main} and using concavity, it comes: 
\begin{align*}
\esp\big[ \lVert \hat\mu_n - \mu \rVert_{\Cku{s}} \big] 
  &\lesssim \frac{(\log J)^d}{J^s} + \Bigg( \sum_{0<\lvert k\rvert_\infty\le J} \frac{\esp\big[ \lvert \hat\mu_n(e_k) - \mu(e_k)\rvert^{2} \big]}{\lvert k\rvert_\infty^{2s}}\Bigg)^{\frac12} \\
  &\lesssim \frac{(\log J)^d}{J^s} + \Bigg( \sum_{0<\lvert k\rvert_\infty\le J} \frac{ \lvert k\rvert^{2\alpha}}{(1-\theta^\alpha)n \lvert k\rvert_\infty^{2s}}\Bigg)^{\frac12} \\  
  &\lesssim \frac{(\log J)^d}{J^s} + \Bigg( \sum_{\ell=1}^J  \frac{\ell^{d-1+2\alpha-2s}}{(1-\theta^\alpha)n}  \Bigg)^{\frac12} % \\
%  &\lesssim \frac{(\log J)^d}{J^s} + \frac{1}{\sqrt{(1-\theta^\alpha)n}} \Big(\sum_{\ell=1}^J \ell^{d-1+2\alpha-2s} \Big)^{\frac12}
\end{align*}
Choose now $\alpha =1/\log J$ so that $\ell^{2\alpha}\lesssim 1$ for all $\ell\in\{1,\dots, J\}$, use $1-\theta^\alpha \ge \alpha(1-\theta)$ and set $\bar n := (1-\theta)n$ to obtain
\begin{equation}
\esp\big[ \lVert \hat\mu_n - \mu \rVert_{\Cku{s}} \big] \lesssim \frac{(\log J)^d}{J^s} + \sqrt{\frac{\log J}{\bar n}} \Big(\sum_{\ell=1}^J \ell^{d-1-2s} \Big)^{\frac12}
\end{equation}

For $s < d/2$, we get:
\begin{align}  
\esp\big[ \lVert \hat\mu_n - \mu \rVert_{\Cku{s}} \big]
  &\lesssim \frac{(\log J)^d}{J^s} + \frac{(\log J)^{\frac12} J^{\frac{d}{2}-s}}{\sqrt{\bar n}}
  \label{eq:Markov-explicit-bound}
\end{align}
Trying to balance the contribution of the two terms, we first see that taking $J \simeq \bar n^{\frac1d}$ would optimize the power of $\bar n$ in the final expression; refining to $J=(\log \bar n)^\beta \bar n^{\frac1d}$, developing and ignoring lower order terms shows that the choice $\beta=2-\frac1d$ optimizes the final power of $\log \bar n$, and we thus set 
\[ J = \big\lfloor (\log \bar n)^{2-\frac1d} \bar n^{\frac1d} \big\rfloor \]

Any large enough $n$ (the bound depending on both $\theta$ and $d$) satisfies the requirement $n \ge 1/(1-\theta^\alpha)$ since the right-hand side is of the order of $\log n$. It then comes:
\[ \esp\big[ \lVert \hat\mu_n - \mu \rVert_{\Cku{s}} \big] \lesssim \frac{(\log \bar n)^{d-2s+\frac sd}}{\bar n^{\frac sd}} \qquad (n \text{ large enough}).\]

For $2s=d$ we get
\[ \esp\big[ \lVert \hat\mu_n - \mu \rVert_{\Cku{s}} \big] \lesssim \frac{(\log J)^d}{J^s} +  \frac{\log J}{\sqrt{\bar n}} \]
and taking $J = \lfloor \bar n^{\frac{1}{2s}} (\log \bar n)^{(d-1)/s} \rfloor$ yields
\[ \esp\big[\wass_1(\hat\mu_n,\mu)\big] \lesssim \frac{\log \bar n}{\sqrt{\bar n}}.\]

Finally, for $s > d/2$ we get
\[ \esp\big[ \lVert \hat\mu_n - \mu \rVert_{\Cku{s}} \big] \lesssim  \frac{(\log J)^d}{J^s} + \frac{(\log J)^{\frac12} }{\sqrt{\bar n}} \]
and taking $J = \lfloor \bar n^{\frac{1}{2s}} (\log \bar n)^{\frac{d}{s+1/2}} \rfloor$ yields
\[\esp\big[ \lVert \hat\mu_n - \mu \rVert_{\Cku{s}} \big] \lesssim \frac{(\log \bar n)^{\frac{d}{2s+1}}}{\sqrt{\bar n}}, \]
ending the proof of Theorem \ref{theomain:Markov}.

%%%%%%%%%%%%%%%%%%%%%%%%%%%%%%%%%%%%%%%%%%%%%%%%%%%%%%%%%%%%%%%
%%%%%%%%%%%%%%%%%%%%%%%%%%%%%%%%%%%%%%%%%%%%%%%%%%%%%%%%%%%%%%%
\section{Concentration near the expectancy}\label{sec:conc}

Let us detail how classical bounded martingale difference methods can be used to prove that the empirical measure concentrates very strongly around its expectancy. When $(X_k)_{k\ge 0}$ are independent identically distributed, this is long-known (see \cite{talagrand1992matching}, and also \cite{weed2017sharp} for more general Wasserstein metrics $\wass_p$, $p\ge1$). In the case of Markov chains, such arguments have been developed notably in \cite{chazotte2009concentration} and, in a dynamical context, \cite{chazottes2012optimal}. Our approach is very similar and thus cannot pretend to novelty, but we write it down to show how to handle functional spaces more general than just Lipschitz and H\"older.

The fundamental result to be used is the Azuma-Hoeffding inequality, which we recall.
\begin{theo*}[Azuma-Hoeffding inequality]
Let $Y$ be a random variable, let
\[\{\varnothing,\Omega\}=\mathscr{B}_0\subset \mathscr{B}_1\subset \dots \subset \mathscr{B}_n = \mathscr{B}(\Omega)\]
be a filtration and for each $k\in\llbracket 1,n\rrbracket$ set $\Delta_k = \esp[Y | \mathscr{B}_k]-\esp[Y | \mathscr{B}_{k-1}]$.
Assume that for all $k$ and some numbers $a_k\in\mathbb{R}$, $c_k>0$ we have $\Delta_k \in[a_k,a_k+c_k]$ almost surely. Then for all $t>0$,
\[\pr\big[ Y\ge \esp[Y] +t \big] \le \exp\Big(-\frac{2t^2}{\sum_k c_k^2} \Big).\]
\end{theo*}

%%%%%%%%%%%%%%%%%%%%%%%%%%%%%%%%%%%%%%%%%%%%%%%%%%%%%%%%%%%%%%%
\subsection{The independent case}

In the case of i.i.d. random variables, the Azuma-Hoeffding inequality famously yields the following concentration inequality.
\begin{theo*}[McDiarmid's inequality]
Let $F:\Omega^n\to\mathbb{R}$ be a function such that for some $c_1,\dots,c_n$ and all $k\in\llbracket 1,n\rrbracket$ and all $(x_1,\dots,x_n,x_k')\in \Omega^{n+1}$ it holds
\[\big\lvert F(x_1,\dots,x_k,\dots,x_n) - F(x_1,\dots,x_k',\dots,x_n) \big\rvert \le c_k.\]
Let $(X_k)_{1\le k\le n}$ be a sequence of independent random variables. Then for all $t>0$ it holds
\[\pr\big[ F(X_1,\dots,X_n) \ge \esp[F(X_1,\dots,X_n)] +t \big] \le \exp\Big(-\frac{2 t^2}{\sum_k c_k^2} \Big).\]
\end{theo*}

Applying this to
\[F(X_1,\dots,X_n) = \lVert \hat\mu_n - \mu \rVert_{\fspace{F}} = \sup_{f\in\fspace{F}} \Big\lvert \frac1n \sum_{k=1}^n f(X_k) -\mu(f) \Big\rvert\]
we can take 
\[c_k=\frac1n \sup_{f\in\fspace{F},x,x'\in\Omega} \lvert f(x)-f(x')\rvert =: \frac1n \osc(\fspace{F})\]
and it comes
\[\pr\big[ F(X_1,\dots,X_n) \ge \esp[F(X_1,\dots,X_n)] +t \big] \le \exp\Big(-\frac{2n t^2}{\osc(\fspace{F})^2}\Big).\]

For example if $\fspace{F}\subset \Lip_1(\Omega)$ (e.g. $\fspace{F}=\Cku{s}$) we have $\osc(\fspace{F}) \le \diam\Omega$; if moreover $\Omega = [0,1]^d$ it thus comes 
\begin{equation}
\pr\Big[ \lVert \hat\mu_n - \mu\rVert_{\fspace{F}} \ge \esp\big[\lVert \hat\mu_n - \mu\rVert_{\fspace{F}} \big] +t \Big] \le \exp\Big(-\frac{2}{d}\cdot n t^2\Big).
\label{eq:conciid}
\end{equation}
This, combined with Theorem \ref{theomain:reg}, yields good concentration estimates.

\begin{coro}
If $(X_k)_{k\ge 0}$ are i.i.d.random variables with law $\mu$, then for all $s\ge 1$, for some constant $C=C(d,s)>0$ (not depending upon $\mu$), all integer $n\ge 2$ and all $M \ge C$ we have:
\begin{itemize}
\item if $s>d/2$
\begin{equation}
\pr\Big[\lVert \hat\mu_n - \mu \rVert_{\Cku{s}} \ge \frac{M}{\sqrt{n}} \Big] \le e^{-\frac2d (M-C)^2};
\end{equation}
\item if $s=d/2$
\begin{equation}
\pr\Big[\lVert \hat\mu_n - \mu \rVert_{\Cku{s}} \ge \frac{M \log n}{\sqrt{n}} \Big] \le e^{-\frac2d (M-C)^2 (\log n)^2};
\end{equation}
\item if $s<d/2$
\begin{equation}
\pr\Big[\lVert \hat\mu_n - \mu \rVert_{\Cku{s}} \ge \frac{M}{n^{\frac{s}{d}}} \Big] \le e^{-\frac2d (M-C)^2 n^{1-2s/d}};
\end{equation}
\end{itemize}
\end{coro}

Similarly, with Theorem \ref{theomain:W1} we can obtain entirely explicit, non-asymptotic concentration bounds.

%%%%%%%%%%%%%%%%%%%%%%%%%%%%%%%%%%%%%%%%%%%%%%%%%%%%%%%%%%%%%%%
\subsection{Markov Chains}

To tackle Markov chains we will need some hypothesis to replace independence; we choose a framework that covers the case of $\wass_1$, but also more general dual metrics $\lVert\cdot\rVert_\fspace{F}$.

Assume that $\Omega$ is endowed with a metric $d$ with finite diameter ($d$ is assumed to be lower-semi-continuous, but not necessarily to induce the given topology on $\Omega$).
We still denote by $\Lip_1(\Omega)$ be the space of functions $\Omega\to\mathbb{R}$ which are $1$-Lipschitz with respect to $d$.

Let $(X_k)_{\ge 0}$ be a Markov chain on $\Omega$
which is exponentially contracting (see the beginning of Section \ref{sec:MC}) with constant $D$ and rate $\theta$, in the metric $d$ instead of the euclidean norm; this can be rewritten in a coupling formulation as follows: for all $x,x'\in\Omega$, all $i,t\in\mathbb{N}$ there are random variables $(X'_k)_{k\ge i}$ with the same law as $(X'_k)_{k\ge i}$ and such that for all $t$:
\[\esp[d(X_{i+t},X'_{i+t}) \mid X_i=x, X'_i=x'] \le D \theta^t d(x,x').\]
Note that the flexibility in the choice of $d$ enables to include uniformly ergodic Markov chains in this framework, simply by taking $d=\one_{\neq}$, i.e. $d(x,y)=0$ if $x=y$ and $d(x,y)=1$ otherwise.

% it holds
%\[\wass_{1,d}(K^t_x,K^t_{x'})\le C(1-\delta)^t d(x,x')\]
%Note that when $d=\one_{\neq}$ (i.e. $d(x,y)=0$ if $x=y$, $1$ otherwise), then $\wass_{1,d}$ is the total variation distance and the above contracting condition is simply uniform ergodicity.
%
%By Kantorovich duality, we can write this condition in two forms:
%\[\sup_{f\in\Lip_1(\Omega)} \big\lvert \esp[f(X_{i+t}) \mid X_i=x] - \esp[f(X_{i+t}) \mid X_i=x'] \big\rvert \le C(1-\delta)^t d(x,x');\]
%or: for each $x,x',i$ there are random variables $(X'_k)_{k\ge i}$ with the same law as $(X'_k)_{k\ge i}$ and such that for all $t$:
%\[\esp[d(X_{i+t},X'_{i+t}) \mid X_i=x, X'_i=x'] \le C(1-\delta)^t d(x,x').\]

Given a multivariate function $\Phi:\Omega^n\to\mathbb{R}^n$, we define as usual the coordinate-wise Lipschitz constants of $\Phi$ by
\[\Lambda_i(\Phi) = \sup_{x_1,\dots,x_n\in\Omega, x'_i\neq x_i} \frac{\lvert \Phi(x_1,\dots, x_i,\dots, x_n)-\Phi(x_1,\dots, x'_i,\dots,x_n) \rvert}{d(x_i,x'_i)}\]
and we say that $\Phi$ is separately Lipschitz if $\Lambda_i(\Phi)<\infty$ for all $i$ (when $d=\one_\neq$, the coordinate-wise Lipschitz constant become the coordinate-wise oscillations).

\begin{theo}\label{theo:conc}
Let $(X_k)_{k\ge1}$ be a Markov chain whose kernel is exponentially contracting with constant $D\ge 1$ and rate $\theta\in(0,1)$, with respect to a lower-semi-continuous distance $d$ on $\Omega$ giving it finite diameter $\diam(\Omega)$.

Let $n\in\mathbb{N}$ and $\Phi :\Omega^n\to\mathbb{R}$ be separately Lipschitz with constants $\Lambda_i(\Phi)\le \Lambda$. Then
\[\pr\Big[ \Phi(X_1,\dots,X_n) \ge \esp[\Phi(X_1,\dots,X_n)] + t \Big] \le  \exp\Big(-\frac{(1-\theta)^2 t^2}{2n D^2 \diam(\Omega)^2\Lambda^2}\Big) \]
\end{theo}

\begin{proof}
We set $X=(X_1,\dots,X_n)$ and $X_{i:j} = (X_i,\dots,X_j)$ (meaning the empty family whenever $j< i$).

We shall apply the Azuma-Hoeffding inequality with the filtration $\mathscr{B}_k=\sigma(X_1^k)$, leaving us with the task of bounding the oscillations $c_k$ of the random variable
\[\Delta_k = \esp[\Phi(X) | X_{1:k}]-\esp[\Phi(X) | X_{1:k-1}].\]

Given an arbitrary $x_{1:k}=(x_1,\dots,x_k)\in\Omega^k$ and $x_k'\in\Omega$ we set
\[V_k(x_{1:k},x'_k) = \esp[\Phi(X) | X_{1:k}=x_{1:k}]-\esp[\Phi(X) | X_{1:k-1}=x_{1:k-1}, X_k=x'_k]\]
so that $c_k = \sup V_k - \inf V_k\le 2\lVert V_k\rVert_\infty$.
Let $(X'_i)_{i\ge k}$ be a copy of $(X_i)_{i\ge k}$ as in the definition of exponential contraction; then
\begin{align*}
V_k(x_{1:k},x'_k) &= \esp\big[\Phi(x_{1:k-1},X_{k:n}) \big| X_k=x_k\big]-\esp\big[\Phi(x_{1:k-1},X'_{k:n}) \big| {X'}_{k}=x'_{k}\big] \\
  &= \sum_{i=k}^{n} \esp\Big[\Phi(x_{1:k-1},X_{k:i},X'_{i+1:n}) -\Phi(x_{1:k-1},X_{k:i-1},X'_{i:n}) \Big| X_k=x_k,  X'_k=x'_k\Big] \\
\lvert V_k(x_1^k,x'_k) \rvert  &\le \sum_{i=k}^{n} \esp\big[\Lambda d(X_i,X'_{i}) \big| X_k=x_k,  X'_k=x'_k\big] \\
  &\le D\Lambda d(x_k,x'_k) \sum_{i=k}^{\infty} \theta^{i-k}  \\
c_k &\le 2C\Lambda \diam(\Omega)/(1-\theta).
\end{align*}
Applying the Azuma-Hoeffding inequality finishes the proof.
\end{proof}

\begin{rema}
The above inequality is probably not optimal; one can expect to improve the rate, either by moving the constant $2$ from the denominator to the numerator, or by replacing $(1-\theta)^2$ by $(1-\theta)$ (probably with another constant).
\end{rema}

As soon as $\fspace{F}\subset \Lip_1(\Omega)$ (e.g. $\fspace{F}=\Cku{s}$), Theorem \ref{theo:conc} applies to 
\[\Phi(X) = \lVert \hat\mu_n -\mu\rVert_{\fspace{F}} = \sup_{f\in\fspace{F}} \frac1n\sum_{k=1}^n f(X_k) -\mu(f)\]
with $\Lambda=\frac1n$, yielding
\begin{equation}
\pr\Big[ \lVert \hat\mu_n -\mu\rVert_{\fspace{F}} \ge \esp\big[\lVert \hat\mu_n -\mu\rVert_{\fspace{F}}\big] + t \Big] \le  \exp\Big(-\frac{(1-\theta)^2}{2 D^2 \diam(\Omega)^2} \cdot nt^2\Big)
\end{equation}
i.e., as in the independent case, subgaussian concentration. Corollary \ref{coromain:conc} follows.

\bibliographystyle{amsalpha}
\bibliography{empiric}

%%%%%%%%%%%%%%%%%%%%%%%%%%%%%%%%%%%%%%%%%%%%%%%%%%%%%%%%%%%%%%%%%%%%%%%
%%%%%%%%%%%%%%%%%%%%%%%%%%%%%%%%%%%%%%%%%%%%%%%%%%%%%%%%%%%%%%%%%%%%%%%
%%%%%%%%%%%%%%%%%%%%%%%%%%%%%%%%%%%%%%%%%%%%%%%%%%%%%%%%%%%%%%%%%%%%%%%
\end{document}